\documentclass[3p,times]{elsarticle}

\usepackage{amssymb}

\usepackage{amsthm}

\usepackage{amsmath}

\usepackage[T1]{fontenc}

\usepackage{mathrsfs}

\biboptions{}

\newtheorem{theorem}{Theorem}[section]
\newtheorem{remark}[theorem]{Remark}

\newtheorem{corollary}[theorem]{Corollary}

\newtheorem{definition}[theorem]{Definition}

\newproof{pf}{Proof}

\begin{document}
	
\begin{frontmatter}
\title{New criteria for boundedness and stability of nonlinear neutral delay differential equations by Krasnoselskii's fixed point theorem} 
\author{Yang Li}
\author{Guiling Chen}
\address{Department of Mathematics,	Southwest Jiaotong University,
		Chengdu 610031, China}	
\begin{abstract}
In this paper, we study boundedness,uniform stability and asymptotic stability of a class of nonlinear neutral delay differential equations by using Krasnoselskii's fixed point theorem. The results obtained in this paper extend and improve the work of Jin and Luo(Nonlinear Anal 68:3307-3315,2008), and Benhadri, Mimia(Differ Equ Dyn Syst 29:3-19,2021). An example is given to illustrate the effectiveness of the proposed results.
\end{abstract}

\end{frontmatter}
\section{Introduction and preliminaries}
Neutral delay differential equations are often used to describe the dynamical systems which depend on present and past states. Practical examples of neutral delay differential systems include biological models of single species growth\cite{ref1-1}, processes including steam or water pipes, heat exchanges\cite{ref1-2},population ecology\cite{ref1-3},and other engineering systems\cite{ref1-2}.

For more than one hundred years, Lyapunov's direct method has been an effective technique for dealing with stability in ordinary and functional differential equations. However, Lyapunov's direct method is not always effective in establishing stability results for a differential equation. There is a variety of difficulties, such as, Lyapunov's direct method requires pointwise conditions while many practical problems don't meet these conditions, a suitable Lyapunov function is not easy to construct, there remain problems with ascertaining limit sets when the equation becomes unbounded or the derivative is not definite. A good news is that Burton and other authors have applied fixed point theory to investigate the stability of deterministic systems and obtained some more applicable conclusions, for example, the works [4-13]. This method possesses the advantage that it can yield existence, uniqueness, and stability in one step. We can easily obtain the existence through the theorem itself, and obtain stability by constructing contradiction. 

Recently, many researchers have studied the boundedness and stability of neutral delay differential equations by applying Krasnoselskii's fixed point theorem. T.A. Burton \cite{ref3} has investigated the boundedness and the stability of the linear equation
\begin{eqnarray*}
	x'(t)=-a(t)+x(t-r_1(t)).
\end{eqnarray*}

T.A. Burton and T.Furumochi \cite{ref4} have studied the boundedness and the asymptotic stability of the following equation
\begin{eqnarray*}
	x'(t)=-a(t)x(t-r_1)+b(t)x^\frac{1}{3}(t-r_2(t)),
\end{eqnarray*}
with $r_1\geq 0$ is a constant and $a \in C(\mathbb{R^+},(0,\infty))$. 

C.H. Jin and J.W. Luo \cite{ref6} have studied the boundedness and stability of the following equation by using the Krasnoselskii's fixed point theorem
\begin{eqnarray*}
	x'(t)=-a(t)x(t-r_1(t))+b(t)x^\frac{1}{3}(t-r_2(t)).
\end{eqnarray*}

Benhadri, M. \cite{ref7} has presented the stability results by using Krasnoselskii's fixed point theorem for the following equation
\begin{eqnarray*}
	x'(t)=-a(t)x(t-r_1(t))+b(t)x'(t-r_1(t))+c(t)x^{\gamma}(t-r_2(t)),\quad t \geq t_0.
\end{eqnarray*} 

Ardjouni,A., Djoudi,A. \cite{ref8} have obtained the boundedness and stability results by using Krasnoselskii's fixed point theory for the following equation
\begin{eqnarray*}
	x'(t)=-a(t)x(t-r_1(t))+b(t)x'(t-r_1(t))+c(t)G(x^{\gamma}(t-r_2(t))).
\end{eqnarray*}

We notice that the results in \cite{ref8} are mainly dependent on the constraint $ \left| \frac{b(t)}{1-r_1'(t)}\right| < 1$. However, there are some interesting examples where the constraint is not satisfied.

It is our aim in this paper to remove this constraint condition and study a general class of nonlinear neutral differential equations with variable delays 
\begin{eqnarray}\label{2}
	x'(t)=-a(t)x(t-r_1(t))+\frac{d}{dt}Q(t,x(t-r_1(t)))+d(t)F(x(t-r_1(t)),x(t-r_2(t)))+c(t)G(x^{\gamma}(t-r_2(t))), \quad t\geq t_0,
\end{eqnarray}
where $\gamma\in (0,1)$ is a quotient with odd positive integer denominator. We assume that $a,b,c\in C(\mathbb{R}^+,\mathbb{R})$, $r_1\in C^2((\mathbb{R^+},\mathbb{R^+})$, $r_2\in C((\mathbb{R^+},\mathbb{R^+})$, $r_1,r_2$ satisfy $t-r_j(t)\to\infty$ as $t\to\infty$, and for each $t \geq t_0$,
\begin{eqnarray*}
	m_j(t_0)=inf\left\{t-r_j(t),t \geq t_0\right\},m(t_0)=inf\left\{m_j(t_0),j=1,2\right\}.
\end{eqnarray*}
An initial condition for the differential equation\eqref{2} is defined by 
\begin{eqnarray}\label{initial}
	x(t)=\psi(t) \quad\text{for }  t\in [m(t_0), t_0],
\end{eqnarray}
where $ \psi(t)\in C([m(t_0),t_0],\mathbb{R}) $.

Let$(\mathcal{X},\left|\cdot\right|_h)$ be the Banach space of continuous functions $\varphi :[m(t_0,\infty) \to \mathbb{R}$ with
\begin{eqnarray*}
	\left|\varphi\right|_h: = \sup_{t \geq t_0}\left|\varphi(t)/h(t)\right| < \infty,
\end{eqnarray*}
for each $t_0 \geq 0$ and $\psi \in C([m(t_0),t_0],\mathbb{R})$ fixed, we define $\mathcal{X}_{\psi}$ as the following space:
\begin{eqnarray*}
	\mathcal{X}_{\psi} = \left\{\varphi \in \mathcal{X}: \left|\varphi(t)\right| \leq 1 \text{ for } t \in [m(t_0),\infty) \text{ and } \varphi(t) = \psi(t) \text{ if } t \in [m(t_0),t_0]\right\}.
\end{eqnarray*}	

We assume that the functions $F,G$ are locally Lipschitz continous, then for $t \geq t_0$, there are constants $k_2,k_3,k_4 >0$ so that if $x,y,z,w \in \mathcal{X_\psi}$ then
\begin{eqnarray}\label{Fk}
	\left|F(x,y)-F(z,y)\right|\leq k_2\Vert x-z\Vert,\quad\text{and}\quad\left|F(x,y)-F(x,w)\right|\leq k_3\Vert y-w\Vert, \quad\text{and}\quad F(0,0)=0.
\end{eqnarray}
and
\begin{eqnarray}\label{Gk}
	\left|G(x)-G(y)\right| \leq k_4\Vert x-y\Vert,\quad\text{and}\quad G(0)=0.
\end{eqnarray}
Moreover, $Q$ is continuous and there exist a function $b(t)\in C(\mathbb{R},\mathbb{R^+})$, such that for all $\varphi_1,\varphi_2 \in \mathcal{X_\psi}$ and for all $t \geq t_0$, we have 
\begin{eqnarray}\label{Qk}
	\left| Q(t,\varphi_1)-Q(t,\varphi_2)\right|\leq b(t)\Vert\varphi_1-\varphi_2\Vert,\quad\text{and}\quad Q(t,0)=0.
\end{eqnarray}

\begin{definition}
For each $(t_0,\psi) \in \mathbb{R^+}\times C([m(t_0),t_0],\mathbb{R})$, a solution of \eqref{2} through $(t_0,\psi)$ is a continuous function $x:[m(t_0),t_0+\rho) \to \mathbb{R}$ for some positive constant $\rho > 0$, such that x satisfies \eqref{2} on $[m(t_0),t_0+\rho) \to \mathbb{R}$, and $x=\psi$ on $[m(t_0),t_0]$. We denote such a solution by $x(t)=x(t,t_0,\psi)$. Besides, we define $\Vert\psi\Vert=\max \left\{\left|\psi\right|: m(t_0) \leq t \leq t_0\right\}$.
\end{definition}

\begin{definition}
	The zero solution of \eqref{2} is said to be 
	
	{\rm(i)} stable if for each $\epsilon>0$ and $t_0 \in \mathbb{R^+}$, there exist a $\delta=\delta(\epsilon,t_0)$ such that for $\psi \in C([m(t_0),t_0],\mathbb{R})$ and $\Vert\psi\Vert\leq\delta$ imply $|x(t,t_0,\psi)|<\epsilon$ for $t \geq t_0$;
	
	{\rm(ii)} uniformly stable if the $\delta $ in (i) is independent of $t_0$;
	
	{\rm(iii)} asymptotically stable if $x(t,t_0,\psi)$ is stable and for any $\epsilon>0$ and $t_0\geq 0$, there exists a $\delta=\delta(\epsilon,t_0)$ such that for $\psi \in C([m(t_0),t_0],\mathbb{R})$ and $\Vert\psi\Vert\leq\delta$ imply that $\lim_{t \to \infty}x(t,t_0,\psi)=0$.
\end{definition}

\begin{theorem}[Burton \cite{ref1}]
For each continuous initial function $\psi :[m(t_0),t_0] \to \mathbb{R}$, there exists a continuous solution $x(t,t_0,\psi)$, which satisfies \eqref{2} on an interval $[t_0,\sigma)$ for some $\sigma > 0$, and $x(t,t_0,\psi)=\psi(t)$, $t \in [m(t_0),t_0]$.
\end{theorem}
We end this section by stating the fixed point theorem that will be applied in this paper. For more details on Krasnoselskii's fixed point theorem, we refer to \cite{ref1}.

\begin{theorem}[Krasnoselskii]
	Let $\mathcal{M}$ be a closed convex non-empty subset of a Banach space $(\mathcal{X},\left\|\cdot\right\|)$. Suppose that $A$ and $B$ map $\mathcal{M}$ into $\mathcal{X}$ such that the following conditions hold
	
	{\rm (i)}$Ax+By \in \mathcal{M} ,\forall x,y \in \mathcal{M}$;
	
	{\rm (ii)}$A$ is continous and $A\mathcal{M}$ is contained in a compact set;
	
	{\rm (iii)}$B$ is a contraction with $\alpha < 1$.
 
    Then there is a $z \in \mathcal{M}$, with $z=Az+Bz$.
\end{theorem}

\section{Existence and boundedness}
In this section, we introduce two auxiliary continuous functions $g(t)$ and $p(t)$ to define appropriate mappings, and use Krasnoselskii's fixed point theorem to present criteria for existence and boundedness of equations \eqref{2} with initial condition \eqref{initial} which can be applied in the case $\left| \frac{b(t)}{1-r_1'(t)} \right| \geq 1$ and $\left|b(t-r_1(t))\right|\geq 1$ as well.

\begin{theorem}\label{th1}
	Consider the nonlinear neutral differential equation (\ref{2}) with the initial condition (\ref{initial}) and suppose the following conditions are satisfied:
	
	{\rm (i)} $ r_1(t) $ is twice differentiable with $ r_1'(t)\neq 1$ for all $t \in [m(t_0), \infty)$,

	{\rm (ii)} there exists a bounded function $ p:[m(t_0), \infty)\to (0, \infty)  $ with $ p(t)=1 $ for $ t\in [m(t_0),t_0] $ such that $ p'(t) $ exists for all $ t \in [m(t_0), \infty)$, and there exists a constant $\alpha \in (0,1)$, and $L_1, L_2 >0$, $b(t),k_2,k_3,k_4$ are denoted in (\ref{Fk})-(\ref{Qk}), and an arbitrary continuous function $ g \in C([m(t_0), \infty), \mathbb{R^+})$ such that for $\left|t_1-t_2\right| \leq 1$, 
\begin{eqnarray}\label{L1}
	\Big|\int_{t_1}^{t_2}\left|c(u)\frac{p^{\gamma}(u-r_2(u)}{p(u)}\right|du\Big|\leq L_1\left|t_1-t_2\right|,
\end{eqnarray}
and
\begin{eqnarray}\label{L2}
	\Big|\int_{t_1}^{t_2}g(u)du\Big|\leq L_2\left|t_1-t_2\right|,
\end{eqnarray}
while for $t\geq t_0$, there exist a constant $\alpha \in (0,1)$ such that 
	\begin{eqnarray}\label{the result 2}
		&&\left|\frac{p(t-r_1(t))}{p(t)}b(t-r_1(t))\right|+\int_{t-r_1(t)}^{t}\left|g(u)-\frac{p'(u)}{p(u)}\right|du\nonumber\\
		&&\quad+\int_{t_0}^{t}e^{-\int_{s}^{t}g(u)du}\left\{\left|\left(g(s-r_1(s))-\frac{p'(s-r_1(s))}{p(s-r_1(s))}\right)(1-r_1'(s))-\frac{a(s)p(s-r_1(s))}{p(s)}\right|\right\}ds\nonumber\\
		&&\quad+\int_{t_0}^{t}e^{-\int_{s}^{t}g(u)du}\left|g(s)\right|\left(\int_{s-r_1(s)}^{s}\left|g(u)-\frac{p'(u)}{p(u)}\right|du\right)ds \nonumber\\
		&&\quad+\int_{t_0}^{t}e^{-\int_{s}^{t}g(u)du}\left|\frac{g(s)p(s)-p'(s)}{p^2(s)}b(s-r_1(s))\right|\left|p(s-r_1(s))\right|ds\nonumber\\
		&&\quad+\int_{t_0}^{t}e^{-\int_{s}^{t}g(u)du}\left|\frac{d(s)}{p(s)}\right|\left|k_2p(s-r_1(s))+k_3p(s-r_2(s))\right|ds\nonumber\\
		&&\quad+k_4\int_{t_0}^{t}e^{-\int_{s}^{t}g(u)du}\left|\frac{c(s)}{p(s)}\right|\left|p^{\gamma}(s-r_2(s))\right|ds \leq \alpha.
	\end{eqnarray}
	
	If $\psi$ is a given continuous initial function which is sufficiently small then there is a solution $x(t,t_0,\psi)$ of \eqref{2} on $\mathbb{R}$ which is bounded.
\end{theorem}

\begin{proof}
Let $z(t)=\psi(t)$ on $t \in [m(t_0),t_0]$ and for $t \geq t_0$, let
\begin{eqnarray}\label{defz(t)}
	x(t)=p(t)z(t).
\end{eqnarray} 
Substitute \eqref{defz(t)} in \eqref{2}, we have
	\begin{eqnarray}\label{dz(t)2}
	z'(t)&=&-\frac{p'(t)}{p(t)}z(t)-\frac{a(t)}{p(t)}p(t-r_1(t))z(t-r_1(t))+\frac{1}{p(t)}\frac{d}{dt}Q(t,p(t-r_1(t))z(t-r_1(t)))\nonumber\\
	&&+\frac{d(t)}{p(t)}F(p(t-r_1(t))z(t-r_1(t)),p(t-r_2(t))z(t-r_2(t)))+\frac{c(t)}{p(t)}G(p^{\sigma}(t-r_2(t))z^{\sigma}(t-r_2(t))).
	\end{eqnarray}
Multiply both sides of \eqref{dz(t)2} by $e^{{\int_{t_0}^{t}}g(s)ds}$ and then integrate from $t_0$ to $t$, we have
\begin{eqnarray*}
	z(t)&=&\psi(t_0)e^{-\int_{t_0}^{t}g(u)du}+\int_{t_0-r_1(t_0)}^{t_0}e^{-\int_{s}^{t}g(u)du}\left(g(s)-\frac{p'(s)}{p(s)}\right)z(s)ds\\
	&&-\int_{t_0-r_1(t_0)}e^{-\int_{t_0}^{t}g(u)du}\frac{a(s)}{p(s)}p(s-r_1(s))z(s-r_1(s))ds\\
	&&+\int_{t_0}^{t}e^{-\int_{t_0}^{t}g(u)du}\frac{1}{p(s)}\frac{d}{ds}Q(s,p(s-r_1(s))z(s-r_1(s)))ds\\
	&&+\int_{t_0}^{t}e^{-\int_{t_0}^{t}g(u)du}\frac{c(s)}{p(s)}F(p(s-r_1(s))z(s-r_1(s)),p(s-r_2(s))z(s-r_2(s)))ds\\
	&&+\int_{t_0}^{t}e^{-\int_{t_0}^{t}g(u)du}\frac{b(s)}{p(s)}G(p^{\sigma}(s-r_2(s))z^{\sigma}(s-r_2(s)))ds.
\end{eqnarray*}
And then perform integration by parts , we can conclude for $t\geq t_0$, 
\begin{eqnarray*}
	z(t)&=&\psi(t_0)e^{-\int_{t_0}^{t}g(u)du}+\int_{t_0-r_1(t_0)}^{t_0}e^{-\int_{s}^{t}g(u)du}d(\int_{s-r_1(s)}^{s}\left(g(u)-\frac{p'(u)}{p(u)}\right)z(u)du)\\
	&&+\int_{t_0}^{t}e^{-\int_{t_0}^{t}g(u)du}\left(g(s-r_1(s))-\frac{p'(s-r_1(s))}{p(s-r_1(s))}\right)(1-r_1'(s))z(s-r_1(s))ds\\
	&&-\int_{t_0-r_1(t_0)}e^{-\int_{t_0}^{t}g(u)du}\frac{a(s)}{p(s)}p(s-r_1(s))z(s-r_1(s))ds\\
    &&+\int_{t_0}^{t}e^{-\int_{t_0}^{t}g(u)du}\frac{1}{p(s)}dQ(s,p(s-r_1(s))z(s-r_1(s)))\\
	&&+\int_{t_0}^{t}e^{-\int_{t_0}^{t}g(u)du}\frac{c(s)}{p(s)}F(p(s-r_1(s))z(s-r_1(s)),p(s-r_2(s))z(s-r_2(s)))ds\\
	&&+\int_{t_0}^{t}e^{-\int_{t_0}^{t}g(u)du}\frac{b(s)}{p(s)}G(p^{\sigma}(s-r_2(s))z^{\sigma}(s-r_2(s)))ds.
\end{eqnarray*}
Thus,
	\begin{eqnarray}\label{z(t)2}
		z(t)&=
		&\left[\psi(t_0)-\int_{t_0-r_1(t_0)}^{t_0}\left(g(s)-\frac{p'(s)}{p(s)}\right)z(s)ds-\frac{1}{p(t_0)}Q(t_0,p(t_0-r_1(t_0))z(t_0-r_1(t_0)))\right]e^{-\int_{t_0}^{t}g(s)ds}\nonumber\\
		&&+\int_{t-r_1(t)}^{t}\left(g(s)-\frac{p'(s)}{p(s)}\right)z(s)ds-\int_{t_0}^{t}e^{-\int_{s}^{t}g(u)du}\left(\int_{s-r_1(s)}^{s}\left(g(u)-\frac{p'(u)}{p(u)}\right)z(u)du\right)g(s)ds\nonumber\\
		&&+\int_{t_0}^{t}e^{-\int_{s}^{t}g(u)du}\left\{\left(g(s-r_1(s))-\frac{p'(s-r_1(s))}{p(s-r_1(s))}\right)(1-r_1'(s))-\frac{a(s)}{p(s)}\right\}z(s-r_1(s))ds\nonumber\\
		&&+\frac{1}{p(t)}Q(t,p(t-r_1(t))z(t-r_1(t)))-\int_{t_0}^{t}e^{-\int_{s}^{t}g(u)du}Q(s,p(s-r_1(s))z(s-r_1(s)))\frac{g(s)p(s)-p'(s)}{p^2(s)}ds\nonumber\\
		&&+\int_{t_0}^{t}e^{-\int_{s}^{t}g(u)du}\frac{b(s)}{p(s)}F(p(s-r_1(s))z(s-r_1(s)),p(s-r_2(s))z(s-r_2(s)))ds\nonumber\\
		&&+\int_{t_0}^{t}e^{-\int_{s}^{t}g(u)du}\frac{c(s)}{p(s)}G(p^{\gamma}(s-r_2(s))z^{\gamma}(s-r_2(s))),
	\end{eqnarray}
For $\alpha \in (0,1)$ , choose $\delta > 0$ is sufficiently small, such that:
\begin{eqnarray}\label{to prove A+B2}
	\left(1+\int_{t_0-r_1(t_0)}^{t_0}\left|g(s)-\frac{p'(s)}{p(s)}\right|ds+b(t_0)\right)e^{-\int_{t_0}^{t}g(s)ds}\delta + \alpha \leq 1.
\end{eqnarray}

Let $\psi : [m(t_0),t_0] \to \mathbb{R}$ be a given small bounded initial function with $\Vert \psi \Vert < \delta$. And let $h:[m(t_0), \infty)\to [1, \infty)$ be any strictly increasing and continuous function with $h(m(t_0)) = 1$, $h(s) \to \infty$ as $ t \to \infty$, such that:
\begin{eqnarray}\label{h(t)2}
	&&\left|\frac{p(t-r_1(t))}{p(t)}b(t-r_1(t))\right|+\int_{t-r_1(t)}^{t}\left|g(s)-\frac{p'(s)}{p(s)}\right|h(s)/h(t)du\nonumber\\
	&&\quad+\int_{t_0}^{t}e^{-\int_{s}^{t}g(u)du}\left\{\left|\left(g(s-r_1(s))-\frac{p'(s-r_1(s))}{p(s-r_1(s))}\right)(1-r_1'(s))-\frac{a(s)p(s-r_1(s))}{p(s)}\right|\right\}h(s)/h(t)ds\nonumber\\
	&&\quad+\int_{t_0}^{t}e^{-\int_{s}^{t}g(u)du}\left|g(s)\right|\left(\int_{s-r_1(s)}^{s}\left|g(u)-\frac{p'(u)}{p(u)}h(u)/h(t)\right|du\right)ds \nonumber\\
	&&\quad+\int_{t_0}^{t}e^{-\int_{s}^{t}g(u)du}\left|\frac{g(s)p(s)-p'(s)}{p^2(s)}\right|\left|p(s-r_1(s))b(s-r_1(s))\right|h(s)/h(t)ds\nonumber\\
	&&\quad+\int_{t_0}^{t}e^{-\int_{s}^{t}g(u)du}\left|\frac{b(s)}{p(s)}\right|\left|k_2p(s-r_1(s))+k_3p(s-r_2(s))\right|h(s)/h(t)ds\leq \alpha.
\end{eqnarray}

We note that to apply Krasnoselskii's fixed point theorem, we need to construct two appropriate mappings. Define the mappings: $A,B: \mathcal{X} \to \mathcal{X}$ by $\varphi \in \mathcal{M}$ implies that
\begin{eqnarray}\label{A2}
	(A\varphi)(t)=\int_{t_0}^{t}e^{-\int_{s}^{t}g(u)du}\frac{c(s)}{p(s)}G(p^{\gamma}(s-r_2(s))\varphi^{\gamma}(s-r_2(s)))ds,
\end{eqnarray}
and
\begin{eqnarray}\label{B2}
	(B\varphi)(t)&=&\left[\psi(t_0)-\int_{t_0-r_1(t_0)}^{t_0}\left(g(s)-\frac{p'(s)}{p(s)}\right)\psi(s)ds-\frac{1}{p(t_0)}Q(t_0,p(t_0-r_1(t_0))\psi(t_0-r_1(t_0)))\right]e^{-\int_{t_0}^{t}g(s)ds}\nonumber\\
	&&+\int_{t-r_1(t)}^{t}\left(g(s)-\frac{p'(s)}{p(s)}\right)\varphi(s)ds-\int_{t_0}^{t}e^{-\int_{s}^{t}g(u)du}\left(\int_{s-r_1(s)}^{s}\left(g(u)-\frac{p'(u)}{p(u)}\right)\varphi(u)du\right)g(s)ds\nonumber\\
	&&+\int_{t_0}^{t}e^{-\int_{s}^{t}g(u)du}\left\{\left(g(s-r_1(s))-\frac{p'(s-r_1(s))}{p(s-r_1(s))}\right)(1-r_1'(s))-\frac{a(s)}{p(s)}\right\}\varphi(s-r_1(s))ds\nonumber\\
	&&+\frac{1}{p(t)}Q(t,p(t-r_1(t))\varphi(t-r_1(t)))-\int_{t_0}^{t}e^{-\int_{s}^{t}g(u)du}Q(s,p(s-r_1(s))\varphi(s-r_1(s)))\frac{g(s)p(s)-p'(s)}{p^2(s)}ds\nonumber\\
	&&+\int_{t_0}^{t}e^{-\int_{s}^{t}g(u)du}\frac{b(s)}{p(s)}F(p(s-r_1(s))\varphi(s-r_1(s)),p(s-r_2(s))\varphi(s-r_2(s)))ds.
\end{eqnarray}

We now shows that $\varphi,\eta \in \mathcal{X}_{\psi}$ implies that $A\varphi+B\eta \in \mathcal{X}_{\psi}$. Now, let $\Vert\cdot\Vert$ be the supremum norm on $[m(t_0),\infty)$ of $\varphi \in \mathcal{X}$ if $\varphi$ is bounded. Note that if $\varphi,\eta \in \mathcal{X}_{\psi}$, then 
\begin{eqnarray}\label{A+B<1,2}
	\left|(A\varphi)(t)+(B\eta)(t)\right|
	&\leq&\left(1+\int_{t_0-r_1(t_0)}^{t_0}\left|g(s)-\frac{p'(s)}{p(s)}\right|ds+b(t_0)\right)e^{-\int_{t_0}^{t}g(s)ds}\Vert\psi\Vert\nonumber\\
	&&\quad+\int_{t-r_1(t)}^{t}\left|g(s)-\frac{p'(s)}{p(s)}\right|ds\Vert\varphi\Vert+\int_{t_0}^{t}e^{-\int_{s}^{t}g(u)du}\left(\int_{s-r_1(s)}^{s}\left|g(u)-\frac{p'(u)}{p(u)}\right|du\right)|g(s)|ds\Vert\varphi\Vert\nonumber\\
	&&\quad+\int_{t_0}^{t}e^{-\int_{s}^{t}g(u)du}\left|\left(g(s-r_1(s))-\frac{p'(s-r_1(s))}{p(s-r_1(s))}\right)(1-r_1'(s))-\frac{a(s)}{p(s)}\right|ds\Vert\varphi\Vert\nonumber\\
	&&\quad+\left|\frac{p(t-r_1(t))}{p(t)}b(t-r_1(t))\right|\Vert\varphi\Vert\nonumber\\
	&&\quad+k_1\int_{t_0}^{t}e^{-\int_{s}^{t}g(u)du}|p(s-r_1(s))b(s-r_1(s))|\left|\frac{g(s)p(s)-p'(s)}{p^2(s)}ds\right|\Vert\varphi\Vert\nonumber\\
	&&\quad+\int_{t_0}^{t}e^{-\int_{s}^{t}g(u)du}\left|\frac{b(s)}{p(s)}\right|\left|k_2p(s-r_1(s))+k_3p(s-r_2(s))\right|ds\Vert\varphi\Vert\nonumber\\
	&&\quad+k_4\int_{t_0}^{t}e^{-\int_{s}^{t}g(u)du}\left| \frac{c(s)}{p(s)}\right|\left| p^{\gamma}(s-r_2(s))\right|ds\Vert\varphi\Vert^{\sigma}\nonumber\\
	&\leq&\left(1+\int_{t_0-r_1(t_0)}^{t_0}\left|g(s)-\frac{p'(s)}{p(s)}\right|ds+b(t_0)\right)e^{-\int_{t_0}^{t}g(s)ds}\delta + \alpha \leq 1.
\end{eqnarray}

Thus, we have that $\left|(A\varphi)(t)+(B\eta)(t)\right|\leq 1$ for $t \in [m(t_0),\infty)$. According to the proof of Theorem 2.1, we have that $A$ is continuous and $A\mathcal{X_\psi}$ resides in a compact set.

Now, we will show that $B$ is a contraction with respect to the norm $|\cdot|_h$. For $\forall \varphi_1, \varphi_2 \in \mathcal{X_\psi}$, we have
\begin{eqnarray*}	
	&&\left|(B\varphi_1)(t)-(B\varphi_2)(t)\right|/h(t)\\
	&&\quad\leq \left|\varphi_1-\varphi_2\right|_h\left|\frac{p(t-r_1(t))}{p(t)}b(t-r_1(t))\right|+\int_{t-r_1(t)}^{t}\left|g(s)-\frac{p'(s)}{p(s)}\right|h(s)/h(t)du\\
	&&\quad\quad+\left|\varphi_1-\varphi_2\right|_h\int_{t_0}^{t}e^{-\int_{s}^{t}g(u)du}\left\{\left|\left(g(s-r_1(s))-\frac{p'(s-r_1(s))}{p(s-r_1(s))}\right)(1-r_1'(s))-\frac{a(s)p(s-r_1(s))}{p(s)}\right|\right\}\\
	&&\quad \quad \quad \times h(s-r_1(s))/h(t)ds\\
	&&\quad\quad+\left|\varphi_1-\varphi_2\right|_h\int_{t_0}^{t}e^{-\int_{s}^{t}g(u)du}\left|g(s)\right|\left(\int_{s-r_1(s)}^{s}\left|g(u)-\frac{p'(u)}{p(u)}\right|h(u)/h(t)du\right)ds\\
	&&\quad\quad+\left|\varphi_1-\varphi_2\right|_h\int_{t_0}^{t}e^{-\int_{s}^{t}g(u)du}\left|\frac{g(s)p(s)-p'(s)}{p^2(s)}\right|\left|p(s-r_1(s))b(s-r_1(s))\right|h(s-r_1(s))/h(t)ds\\
	&&\quad\quad+\left|\varphi_1-\varphi_2\right|_h\int_{t_0}^{t}e^{-\int_{s}^{t}g(u)du}\left|\frac{b(s)}{p(s)}\right|\left( \left|k_2p(s-r_1(s))\right|h(s-r_1(s))/h(t)+\left|k_3p(s-r_2(s))\right|h(s-r_2(s))/h(t)\right)ds\\
	&&\quad\leq \alpha\left|\varphi_1-\varphi_2\right|_h.
\end{eqnarray*}

The conditions of Krasnoselskii's theorem are satisfied and there is a fixed point $z(t)$ such that $Az+Bz=z$, which is a solution of \eqref{dz(t)2} with $z(s)=\psi(s)$ on $s \in [m(t_0),t_0]$ and $\left|z(t,t_0,\psi)\right| \leq 1$ for $t \in [m(t_0),\infty)$. Since there exists a bounded function $p: [m(t_0),\infty) \to  (0,\infty)$ with $p(t)=1$ for $t \in [m(t_0),t_0]$, by the hypotheses \eqref{defz(t)} and from the above arguments we deduce that there exists a solution $x$ of \eqref{1} with $x=\psi$ on $[m(t_0),t_0]$ satisfies $x(t,t_0,\psi)$ is bounded for all $t \in [m(t_0),\infty)$. This completes the proof.
\end{proof}
\begin{remark}
The reason why we choose Krasnoselskii's theorem instead of Contraction principle is that we can notice that there exist a term $\varphi^{\gamma}$ in $A\varphi$, which we can't obtain $\left|\varphi_1^{\gamma}-\varphi_2^{\gamma}\right|\leq\left|\varphi_1-\varphi_2\right|$. Thus, we choose  Krasnoselskii's theorem to aviod $\left|\varphi_1^{\gamma}-\varphi_2^{\gamma}\right|$, but we go to verify $A$ is continous and $A\mathcal{M}$ is contained in a compact set.
\end{remark}
\begin{remark}
	We note that  $x(t,t_0,\psi)$ is bounded, but we can't get $x(t,t_0,\psi)\leq 1$. That's because $x(t,t_0,\psi)$ is not only related to $z(t)$, but also to $p(t)$. We just have that $\left|z(t,t_0,\psi)\right| \leq 1$, and the range of $p(t)$ is not determined. Thus, we can control the range of $x(t)$ by controlling the range of $p(t)$.
\end{remark}
Let $p(t)=1$, we obtain the following corollary.
\begin{corollary}\label{cor1}
	Suppose the conditions (i) and (ii) in Theorem \ref{th1} hold, \eqref{the result 2} is replaced by 
	\begin{eqnarray*}
		&&\left|b(t-r_1(t))\right|+\int_{t-r_1(t)}^{t}\left|g(u)\right|du+\int_{t_0}^{t}e^{-\int_{s}^{t}g(u)du}\left\{\left|g(s-r_1(s))(1-r_1'(s))-a(s)\right|\right\}ds\\
		&&\quad+\int_{t_0}^{t}e^{-\int_{s}^{t}g(u)du}\left|g(s)\right|\left(\int_{s-r_1(s)}^{s}\left|g(u)\right|du\right)ds+\int_{t_0}^{t}e^{-\int_{s}^{t}g(u)du}\left|g(s)b(s-r_1(s))\right|ds\\
		&&\quad+\int_{t_0}^{t}e^{-\int_{s}^{t}g(u)du}\left|d(s)\right|(k_2+k_3)ds+k_4\int_{t_0}^{t}e^{-\int_{s}^{t}g(u)du}\left|c(s)\right|ds \leq \alpha.
	\end{eqnarray*} 
	If $\psi$ is a given continuous initial function which is sufficiently small then there is a solution $x(t,t_0,\psi)$ of \eqref{2} on $\mathbb{R}$ which is bounded.
\end{corollary}

\begin{remark}
	From Corollary \ref{cor1}, we can easily see that we mainly rely on $|b(t-r_1(t))|< 1$ in the absence of $p(t)$ to obatin the existence of the zero solution of \eqref{2}. However, many actual situations can't  satify this condition. Our results relax the limitations on the neutral term through introducing $p(t)$.
\end{remark}

\begin{corollary}\label{cor2}
Consider the nonlinear neutral differential equation 
\begin{eqnarray}\label{1}
x'(t)=-a(t)x(t-r_1(t))+b(t)x'(t-r_1(t))+c(t)G(x^{\gamma}(t-r_2(t))), \quad t\geq t_0,
\end{eqnarray}
 with initial condition \eqref{initial} and suppose the following conditions are satisfied:

{\rm (i)} $ r_1(t) $ is twice differentiable with $ r_1'(t)\neq 1$ for all $t \in [m(t_0), \infty)$,

{\rm (ii)} there exists a bounded function $ p:[m(t_0), \infty)\to (0, \infty)  $ with $ p(t)=1 $ for $ t\in [m(t_0),t_0] $ such that $ p'(t) $ exists for all $ t \in [m(t_0), \infty)$, and there exists a constant $\alpha \in (0,1)$, and $L_1, L_2 >0$, $k_4$ is denoted in (\ref{Gk}), and an arbitrary continuous function $ g \in C([m(t_0), \infty), \mathbb{R^+})$ such that for $\left|t_1-t_2\right| \leq 1$,$\eqref{L1}$and $\eqref{L2}$hold,
while for $t\geq t_0$, there exist a constant $\alpha \in (0,1)$
\begin{eqnarray}\label{the result}
		&&\left|\frac{p(t-r_1(t))}{p(t)}\frac{b(t)}{1-r_1'(t)}\right|+\int_{t-r_1(t)}^{t}\left|g(u)-\frac{p'(u)}{p(u)}\right|du\nonumber\\
		&&\quad+\int_{t_0}^{t}e^{-\int_{s}^{t}g(u)du}\left\{\left|-\overline\mu(s)+\left(g(s-r_1(s))-\frac{p'(s-r_1(s))}{p(s-r_1(s))}\right)(1-r_1'(s))-\overline\beta(s)\right|\right\}ds\nonumber\\
		&&\quad+\int_{t_0}^{t}e^{-\int_{s}^{t}g(u)du}\left|g(s)\right|\left(\int_{s-r_1(s)}^{s}\left|g(u)-\frac{p'(u)}{p(u)}\right|du\right)ds \nonumber\\
		&&\quad+k_4\int_{t_0}^{t}e^{-\int_{s}^{t}g(u)du}\left|\frac{c(s)}{p(s)}\right|\left|p^{\gamma}(s-r_2(s))\right|ds \leq \alpha,
\end{eqnarray}
where
\begin{eqnarray*}\label{mu}
	\overline{\mu}(s)=\frac{a(s)p(s-r_1(s))-b(s)p'(s-r_1(s))}{p(s)},
\end{eqnarray*}

\begin{eqnarray*}
	c(s)=\frac{p(s-r_1(s))}{p(s)}\frac{b(s)}{(1-r_1'(s))},
\end{eqnarray*}

\begin{eqnarray*}\label{beta}
	\overline{\beta}(s)=g(s)c(s)+c'(s).
\end{eqnarray*}

If $\psi$ is a given continuous initial function which is sufficiently small then there is a solution $x(t,t_0,\psi)$ of \eqref{1} on $\mathbb{R}$ which is bounded.
\end{corollary}
Let $p(t)=1$, consider the nonlinear neutral differential equation \eqref{1} with the initial condition \eqref{initial},we have the following corollary.
\begin{corollary}\label{abd's result}
Let \eqref{L1} and \eqref{L2} hold and \eqref{the result} be replaced by
\begin{eqnarray*}
	&&\left|\frac{b(t)}{1-r_1'(t)}\right|+\int_{t-r_1(t)}^{t}g(u)du+\int_{t_0}^{t}e^{-\int_{s}^{t}g(u)du}\left\{\left|-\overline\mu(s)+g(s-r_1(s)(1-r_1'(s))-\overline\beta(s)\right|\right\}ds\\
	&&\quad+\int_{t_0}^{t}e^{-\int_{s}^{t}g(u)du}\left|g(s)\right|\left(\int_{s-r_1(s)}^{s}g(u)du\right)ds+k_4\int_{t_0}^{t}e^{-\int_{s}^{t}g(u)du}\left|c(s)\right|ds \leq \alpha.
\end{eqnarray*}

If $\psi$ is a given continuous initial function which is sufficiently small then there is a solution $x(t,t_0,\psi)$ of \eqref{1} on $\mathbb{R}$ which is bounded.
\end{corollary}

\begin{remark}
Corollary \ref{abd's result} is Theorem 2.1 in \cite{ref7}. Thus, Corollary \ref{cor2} improves Theorem 2.1 in \cite{ref7}.
\end{remark}

Let $G(x^{\gamma}(t-r_2(t))) = x^{\gamma}(t-r_2(t))$ in equation \eqref{1}, we have the following corollary.
\begin{corollary}\label{mima'result}

Let \eqref{L1} and \eqref{L2} hold and \eqref{the result} be replaced by
\begin{eqnarray*}
		&&\left|\frac{p(t-r_1(t))}{p(t)}\frac{b(t)}{(1-r_1'(t))}\right|+\int_{t-r_1(t)}^{t}\left|g(u)-\frac{p'(u)}{p(u)}\right|du\\
		&&\quad+\int_{t_0}^{t}e^{-\int_{s}^{t}g(u)du}\Bigg\{\left|-\overline\mu(s)+(g(s-r_1(s))-\frac{p'(s-r_1(s))}{p(s-r_1(s))})(1-r_1'(s))-\overline\beta(s)\right|\Bigg\}ds\\			&&\quad+\int_{t_0}^{t}e^{-\int_{s}^{t}g(u)du}\left|g(s)\right|\left(\int_{s-r_1(s)}^{s}\left|g(u)-\frac{p'(u)}{p(u)}\right|du\right)ds \\
		&&\quad+\int_{t_0}^{t}e^{-\int_{s}^{t}g(u)du}\left|\frac{c(s)}{p(s)}\right|\left|p^{\gamma}(s-r_2(s))\right|ds\leq \alpha.
\end{eqnarray*}

If $\psi$ is a given continuous initial function which is sufficiently small then there is a solution $x(t,t_0,\psi)$ of \eqref{1} on $\mathbb{R}$ which is bounded.
\end{corollary}

\begin{remark}
Corollary \ref{mima'result} is Theorem 2.1 in \cite{ref8}. Therefore, our result is a generalization of the result in \cite{ref8}.
\end{remark}

\section{Uniform stability and Asymptotic stability}
In this section, we present some new sufficient conditions for uniform stability and asymptotic stability of the zero solution by using Krasnoselskii's fixed point theorem.

\begin{theorem}
Let the conditions \eqref{L1},\eqref{L2}, \eqref{the result 2} hold and assume that 
\begin{eqnarray}\label{A to 0}
\int_{t_0}^{t}e^{-\int_{s}^{t}g(u)du}\left|\frac{c(s)}{p(s)}p^{\gamma}(s-r_2(s))\right|ds \to 0 \text{ as } t \to \infty,
\end{eqnarray}

\begin{eqnarray}\label{g(s) to M}
  \int_{t_0}^{t} g(s)ds \to \infty \text{ as } t \to \infty,
\end{eqnarray}

\begin{eqnarray}\label{K}
   \sup_{0\leq t_1\leq t_2}|\Phi|\leq K<\infty,
\end{eqnarray}
where
\begin{eqnarray*}
	\Phi(t_2,t_1)=e^{-\int_{t_1}^{t_2}g(u)du}.
\end{eqnarray*}
Then the zero solution of \eqref{2} is uniformly stable and asymptotically stable.
\end{theorem}

\begin{proof}
	All of the calculations in the proof of Theorem \ref{th1} hold with $h(t)=1$ when $\left|\cdot\right|_h$ is replaced by the supremum norm $\Vert \cdot \Vert $. Besides, Let
	\begin{eqnarray*}
		\mathcal{X_\psi}=\left\{\varphi \in \mathcal{X}: \left|\varphi(t)\right| \leq 1 \text{ for } t \in [m(t_0),\infty) \text{ and } \varphi(t) = \psi(t) \text{ if } t \in [m(t_0),t_0] \text{ and } \varphi \to 0 \text{ as } t \to \infty\right\}.
	\end{eqnarray*}
	
	For $\varphi \in X_\psi$, we can have $A\varphi \to 0$ as $t \to \infty$ by (\ref{A to 0}), and $B\varphi \to 0$ as $t \to \infty$ by (\ref{g(s) to M}).

	Since $A\mathcal{X_\psi}$ has been shown to be equicontinuous, $A$ maps $\mathcal{X_\psi}$ into a compact subset of $\mathcal{X_\psi}$ and $B$ is a contraction, by Krasnoselskii's theorem, there is a $z \in \mathcal{X_\psi}$ with $Az+Bz=z$. As $z \in \mathcal{X_\psi}$, $z(t,t_0,\psi) \to 0$ as $t \to \infty$. Since $p(t)$ is a positive bounded function, by hypotheses \eqref{defz(t)} and from the above arguments we deduce that there exists a solution $x \in X_{\psi} $ of \eqref{1} with $x(t,t_0,\psi) \to 0$ as $t \to \infty$.

	To obtain uniform stability and asymptotic stability, we need to show that the zero solution of \eqref{dz(t)2} is stable. Let $\epsilon >0$ be given and choose $\delta >0 (\delta < \epsilon)$ satisfying $2\delta K+\alpha \epsilon \leq \epsilon$, in other words $\delta < \min{\epsilon,(1-\alpha)\epsilon/2K}$. Notice that K is defined in (\ref{K}) is independent of $t_0$, thus so is $\delta$. This will give us the uniformly stability.

	For $\Vert\psi\Vert\leq \delta$, we claim that $|z(t)|\leq\epsilon$ for all $t\geq t_0$. If $z(t)=z(t,t_0,\psi)$ is the unique solution corresponding to the initial condition $\psi$, with $\Vert\psi\Vert \leq \delta$. Suppose that there exists a $t'>t_0$ such that $|z(t')|=\epsilon$. Let
	\begin{eqnarray*}
		t^*=\inf\left\{t':|z(t')|=\epsilon\right\}.
	\end{eqnarray*}
	Then we have
	\begin{eqnarray*}
		\left|z(t^*)\right|
		&\leq& \Vert\psi\Vert
		\left(1+\int_{t_0-r_1(t_0)}^{t_0}\left|g(s)-\frac{p'(s)}{p(s)}\right|ds+b(t_0)\right)e^{-\int_{t_0}^{t^*}g(s)ds}\\  
		&&+\epsilon\Bigg\{\left|\frac{p(t^*-r_1(t^*))}{p(t^*)}b(t^*-r_1(t^*))\right|+\int_{t^*-r_1(t^*)}^{t^*}\left|g(u)-\frac{p'(u)}{p(u)}\right|du\\
		&&+\int_{t_0}^{t^*}e^{-\int_{s}^{t^*}g(u)du}\left\{\left|\left(g(s-r_1(s))-\frac{p'(s-r_1(s))}{p(s-r_1(s))}\right)(1-r_1'(s))-\frac{a(s)p(s-r_1(s))}{p(s)}\right|\right\}ds\\
		&&+\int_{t_0}^{t^*}e^{-\int_{s}^{t^*}g(u)du}\left|g(s)\right|\left(\int_{s-r_1(s)}^{s}\left|g(u)-\frac{p'(u)}{p(u)}\right|du\right)ds\\
		&&+\int_{t_0}^{t^*}e^{-\int_{s}^{t^*}g(u)du}\left|\frac{g(s)p(s)-p'(s)}{p^2(s)}\right|\left|p(s-r_1(s))b(s-r_1(s))\right|ds\\
		&&+\int_{t_0}^{t^*}e^{-\int_{s}^{t^*}g(u)du}\left|\frac{c(s)}{p(s)}\right|\left|k_2p(s-r_1(s))+k_3p(s-r_2(s))\right|ds+k_4\int_{t_0}^{t^*}e^{-\int_{s}^{t^*}g(u)du}\left|\frac{c(s)}{p(s)}\right|\left|p^{\gamma}(s-r_2(s))\right|ds\Bigg\}\\
		&\leq& 2\delta K+\alpha \epsilon < \epsilon,
	\end{eqnarray*}
	which contradicts the defination of $t^*$. Thus, $\left|z(t)\right| < \epsilon$ for all $t \geq t_0$, and the zero solution of \eqref{dz(t)2} is stable. Hence the solution of \eqref{dz(t)2} is uniformly stable, and since $z(t)$ converges to zero as $t \to \infty$, we get the asymptotic stability. Since $p(t)$ is a positive bounded function, from the above arguments we obtain that the zero solution of \eqref{1} is uniformly stable and asymptotic stable. The proof is complete.
\end{proof}

\begin{corollary}\label{thm2}
 Let \eqref{L1},\eqref{L2},\eqref{the result} and \eqref{A to 0}-\eqref{K}hold,
Then  the zero solution of \eqref{1} is uniformly stable and asymptotically stable.
\end{corollary}
\begin{remark}
The stability of solution of \eqref{1} and \eqref{2} are related to the constant $K=\sup_{0\leq t_1\leq t_2}\left| e^{-\int_{t_1}^{t_2}g(u)du}\right|$, which depends on the initial moment $t_0$. We notice that the key to obtain the uniform stability is making $K$ independent of $t_0$, as we do in our results.
\end{remark}

\section{An example}
We consider a class of nonlinear neutral delay differential
equation as follows:
\begin{eqnarray}\label{ex1}
	x'(t)=-a(t)x(t-r_1(t))+b(t)x'(t-r_1(t))+c(t)G(x^{\gamma}(t-r_2(t))), \quad t\geq0,
\end{eqnarray}
where $b(t)=\frac{1}{7}\sin(t)$, $r_1(t)=0.2t$, $r_2(t)=0.2t$, $c(t)=\frac{1}{100}\frac{(0.8t+0.2)^{\frac{1}{3}}}{(t+0.1)(t+0.2)}$, $G(x)=\sin(x)$ and $a(t)$ satisfies:
\begin{eqnarray*}
	\left|-\overline\mu(t)+\left(g(t-r_1(t))-\frac{p'(t-r_1(t))}{p(t-r_1(t))}\right)(1-r_1'(t))-\overline\beta(t)\right|\leq \frac{0.015}{t+0.1},
\end{eqnarray*}
where $\overline\mu(t)$ and $\overline\beta(t)$ are defined in Corollary 2.6. 
\begin{proof}
Choose $p(t)=\frac{1}{t+0.2}$, and $g(t)=\frac{0.1}{t+0.1}$, then we have
\begin{eqnarray*}
	&&\left|\frac{p(t-r_1(t))}{p(t)}\frac{b(t)}{1-r_1'(t)}\right|=\left|\frac{t+0.2}{0.8t+0.2}\times\frac{1}{7}\sin(t)\times\frac{1}{0.8}\right|\leq 0.2231,\\
	&&\int_{t-r_1(t)}^{t}\left|g(s-\frac{p'(s)}{p(s)}\right|ds=\int_{0.8t}^{t}\left|\frac{0.1}{s+0.1}+\frac{1}{s+0.2}\right|ds\leq 0.2449,\\
	&&\left|\int_{t_0}^{t}e^{-\int_{s}^{t}g(u)du}\Bigg\{\left|-\overline\mu(s)+(g(s-r_1(s))-\frac{p'(s-r_1(s))}{p(s-r_1(s))})(1-r_1'(s))-\overline\beta(s)\right|\Bigg\}ds\right|\leq\left|\int_{t_0}^{t}e^{-\int_{s}^{t}g(u)du}\frac{0.015}{s+0.1}ds\right|\leq 0.15,\\
	&&\left|\int_{t_0}^{t}e^{-\int_{s}^{t}g(u)du}\left|g(s)\right|\left(\int_{s-r_1(s)}^{s}\left|g(u)-\frac{p'(u)}{p(u)}\right|du\right)ds\right|\leq 0.2449,\\
	&&k_4\left|\int_{t_0}^{t}e^{-\int_{s}^{t}g(u)du}\left|\frac{c(s)}{p(s)}\right|\left|p^{\gamma}(s-r_2(s))\right|ds\right|\leq0.1\left|\int_{t_0}^{t}e^{-\int_{s}^{t}g(u)du}\frac{0.1}{s+0.1}ds\right|\leq 0.1.
\end{eqnarray*}
Let $\alpha=0.2231+0.2449+0.15+0.2449+0.1$, it's easy to see that all the conditions of Corollary 2.6 hold for $\alpha\leq 0.973<1$.
\end{proof} 
Then Corollary 2.6 implies the zero solution of \eqref{ex1} is bounded. Moreover, we can obtain that the zero solution of \eqref{ex1} is uniformly stable and asymptotically stable by Corollary 3.2.

{}

\end{document}